\documentclass[a4paper,12pt,leqno]{amsart}
\usepackage{amsmath,amsthm,amssymb,latexsym,epsfig,graphicx,subfigure}
\numberwithin{equation}{section}

\newtheorem{thm}{Theorem}[section]

\theoremstyle{definition}

\theoremstyle{definition}

\newcommand{\Rn}{\mathbb{R}^{n}}

\newcommand{\R}{\mathbb{R}}

\renewcommand{\H}{\mathbb{H}}
\renewcommand{\H}{\mathbb{H}}

\renewcommand{\H}{\mathcal{H}}

\newcommand {\grtrsim} {\ {\raise-.5ex\hbox{$\buildrel>\over\sim$}}\ }
\newcommand{\e}{\varepsilon}

\newcommand{\khii}{\text{\lower -.4ex\hbox{$\chi$}}}
\DeclareMathOperator{\spt}{spt}

\renewcommand{\a}{\alpha}

\newcommand{\restrict}{\begin{picture}(12,12)
                       \put(2,0){\line(1,0){8}}
                       \put(2,0){\line(0,1){8}}
                      \end{picture}}

\begin{document}
\title {Hausdorff dimension of plane sections and general intersections}
\author{Pertti Mattila}


\begin{abstract} 
This paper extends some results of \cite{M9} and \cite{M1}, in particular, removing assumptions of positive lower density. We give conditions on a general family $P_{\lambda}:\R^n\to\R^m, \lambda \in \Lambda,$ of orthogonal projections  which guarantee that the Hausdorff dimension formula $\dim A\cap P_{\lambda}^{-1}\{u\}=s-m$  holds generically for measurable sets $A\subset\Rn$ with positive and finite $s$-dimensional Hausdorff measure, $s>m$. As an application we prove for measurable sets $A,B\subset\Rn$ with positive $s$- and $t$-dimensional measures that if $s + (n-1)t/n > n$, then   $\dim A\cap (g(B)+z) \geq s+t - n$ for almost all rotations $g$ and  for positively many $z\in\Rn$. We shall also give an application on the estimates of the dimension of the set of exceptional rotations.
\end{abstract}

\maketitle

\section{introduction}

As in \cite{M9}, let $P_{\lambda}:\Rn\to\R^m, \lambda\in\Lambda,$ be orthogonal projections, where $\Lambda$ is a compact metric space.  Suppose that $\lambda\mapsto P_{\lambda}x$ is  continuous for every $x\in\Rn$. Let also $\omega$ be a finite non-zero Borel measure on $\Lambda$. 
These assumptions are just to guarantee that the measurability of the various functions appearing later can easily be checked (left to the reader) and that the forthcoming applications of Fubini's theorem are legitimate.  

We shall first prove the following two theorems. There and later we shall identify absolutely continuous measures with their Radon-Nikodym derivatives. For the notation, see Section \ref{prel}.

\begin{thm}\label{level}
Let $s>m$ and $p>1$. Suppose that $P_{\lambda\sharp}\mu\ll\mathcal L^m$ for $\omega$ almost all $\lambda\in \Lambda$ and that there exists a positive number $C$ such that

\begin{equation}\label{L2}
\iint P_{\lambda\sharp}\mu(u)^p\,d\mathcal L^mu\,d\omega\lambda < C
\end{equation}
whenever $\mu\in\mathcal M(B^n(0,1))$ is such that  $\mu(B(x,r))\leq r^s$ for $x\in \Rn, r>0$. 

If $A\subset\R^n$ is $\mathcal H^s$ measurable, $0<\mathcal H^s(A)<\infty$ and $\theta^s_{\ast}(A,x)>0$ for $\mathcal H^s$ almost all $x\in A$, then for $\mathcal H^s\times\omega$ almost all $(x,\lambda)\in A\times\Lambda$,
\begin{equation}\label{e11}
\dim P_{\lambda}^{-1}\{P_{\lambda}x\}\cap A =s-m,
\end{equation}
and for $\omega$ almost all $\lambda\in \Lambda$,
\begin{equation}\label{e12}
\mathcal L^m(\{u\in\R^m: \dim P_{\lambda}^{-1}\{u\}\cap A = s-m\}) > 0.
\end{equation}
\end{thm}

\begin{thm}\label{levela}
Let $s>m$ and $p>1$. Suppose that $P_{\lambda\sharp}\mu\ll\mathcal L^m$ for $\omega$ almost all $\lambda\in \Lambda$ and that there exists a positive number $C$ such that 

\begin{equation}\label{L2a}
\iint P_{\lambda\sharp}\mu(u)^p\,d\mathcal L^mu\,d\omega\lambda < C\mu(B^n(0,1))
\end{equation}
whenever $\mu\in\mathcal M(B^n(0,1))$ is such that  $\mu(B(x,r))\leq r^s$ for $x\in \Rn, r>0$. 

If $A\subset\R^n$ is $\mathcal H^s$ measurable and $0<\mathcal H^s(A)<\infty$, then 
for $\mathcal H^s\times\omega$ almost all $(x,\lambda)\in A\times\Lambda$,
\begin{equation}\label{e17}
\dim P_{\lambda}^{-1}\{P_{\lambda}x\}\cap A =s-m,
\end{equation}
and for $\omega$ almost all $\lambda\in \Lambda$,
\begin{equation}\label{e18a}
\mathcal L^m(\{u\in\R^m: \dim P_{\lambda}^{-1}\{u\}\cap A = s-m\}) > 0.
\end{equation}
\end{thm}

For $p=2$ Theorem \ref{level}  was proved in \cite{M9} and Theorem \ref{levela} follows by the same argument. For $p>1$ the proofs of both theorems have similar strategy but as an essential new ingredient an argument of Harris from \cite{H} is used, see Section \ref{dimension}. 


Theorem \ref{levela} gives a general version of Marstrand's section theorem. For discussion and references for related results, see \cite[Chapter 6]{M6}.

We shall also give a version of Theorem \ref{levela}, Theorem \ref{levelprod}, for product sets and measures and use it to prove the following intersection theorem:

\begin{thm}\label{inter}
Let $s,t>0$ with $s+(n-1)t/n>n$. If $A\subset\R^n$ is $\mathcal H^s$ measurable with $0<\mathcal H^s(A)<\infty$ and $B\subset\R^n$ is $\mathcal H^t$ measurable with $0<\mathcal H^t(B)<\infty$, then for $\mathcal H^s\times\mathcal H^t\times\theta_n$ almost all $(x,y,g)\in A\times B\times O(n)$,
\begin{equation}\label{e5}
\dim  A\cap (g(B-y)+x) \geq s+t-n,
\end{equation}
and for $\theta_n$ almost all $g\in O(n)$,
\begin{equation}\label{e6}
\mathcal L^n(\{z\in\R^n:\dim A\cap (g(B)+z) \geq s+t-n\}) > 0.
\end{equation}
\end{thm}

In \cite{M9} this was proved using Theorem \ref{level}, with $p=2$, and assuming that $A$ and $B$ have positive lower densities. Then also the equality holds in \eqref{e5} and \eqref{e6}. In general the opposite inequality can fail very badly, see \cite{F}.

Theorem \ref{inter} follows immediately from Theorem \ref{levelprod} once the condition \eqref{L3} is verified for the related projections $P_g, P_g(x,y)=x-g(y), x,y\in\Rn, g\in O(n)$, see Section \ref{intersections}. This was already observed in the arXiv version of \cite{M10}, but not in the journal.

Previously the same conclusion was obtained in \cite{M3} under the hypothesis $s+t>n, s>(n+1)/2$. Notice that this and the assumption $s+(n-1)t/n>n$ overlap but neither is implied by the other. I believe that that sole condition $s+t>n$ should suffice. This would be optimal. See \cite[Chapter 7]{M6} and \cite{M10} for discussions and references for such intersection problems.

We shall also use the present method to improve an estimate from \cite{M1} for the dimension of the set of exceptional $g\in O(n)$, see Theorem \ref{exinter}.



\section{Preliminaries}\label{prel}

We denote by $\mathcal L^n$ the Lebesgue measure in the Euclidean $n$-space $\Rn, n\geq 2,$ and by $\sigma^{n-1}$ the surface measure on the unit sphere $S^{n-1}$. The closed ball with centre $x\in\Rn$ and radius $r>0$ is denoted by $B(x,r)$ or $B^n(x,r)$.The $s$-dimensional Hausdorff measure is $\H^s$ and the Hausdorff dimension is $\dim$. 
The orthogonal group of $\Rn$ is $O(n)$ and its Haar probability measure is $\theta_n$. For $A\subset\Rn$  we denote by $\mathcal M(A)$ the set of non-zero finite Borel measures $\mu$ on $\Rn$ with compact support $\spt\mu\subset A$. The Fourier transform of $\mu$ is defined by
$$\widehat{\mu}(x)=\int e^{-2\pi ix\cdot y}\,d\mu y,~ x\in\Rn.$$

For $0<s<n$ the $s$-energy of $\mu\in\mathcal M(\Rn)$ is 
\begin{equation}\label{eq10}
I_s(\mu)=\iint|x-y|^{-s}\,d\mu x\,d\mu y=c(n,s)\int|\widehat{\mu}(x)|^2|x|^{s-n}\,dx.
\end{equation} 
The second equality is a consequence of Parseval's formula and the fact that the distributional Fourier transform of the Riesz kernel $k_s, k_s(x)=|x|^{-s}$, is a constant multiple of $k_{n-s}$, see, for example, \cite{M5}, Lemma 12.12, or \cite{M6}, Theorem 3.10. These books contain most of the background material needed in this paper.

Notice that if $\mu\in\mathcal M(\Rn)$  satisfies the Frostman condition $\mu(B(x,r))\leq r^s$ for all $x\in\Rn, r>0$, then $I_t(\mu)<\infty$ for all $t<s$. We have for any Borel set $A\subset\Rn$ with $\dim A > 0$, cf. Theorem 8.9 in \cite{M5},
\begin{equation}\label{eq3}
\begin{split}
\dim A&=\sup\{s:\exists \mu\in\mathcal M(A)\ \text{such that}\ \mu(B(x,r))\leq r^s\ \text{for all}\ x\in\Rn, r>0\}\\
&=\sup\{s:\exists \mu\in\mathcal M(A)\ \text{such that}\ I_s(\mu)<\infty\}.
\end{split}
\end{equation}

We shall denote by $f_{\#}\mu$ the push-forward of a measure $\mu$  under a map  $f: f_{\#}\mu(A)= \mu(f^{-1}(A))$. The restriction of $\mu$ to a set $A$ is defined by $\mu\restrict A(B)=\mu(A\cap B)$. The notation $\ll$ stands for absolute continuity.

The lower and upper $s$-densities of $A\subset\R^n$ are defined by

$$\theta^s_{\ast}(A,x)=\liminf_{r\to 0}(2r)^{-s}\mathcal H^s(A\cap B(x,r)),$$ 
$$\theta^{\ast s}(A,x)=\limsup_{r\to 0}(2r)^{-s}\mathcal H^s(A\cap B(x,r)).$$ 
If $\mathcal H^s(A)<\infty$, we have by \cite{M5}, Theorem 6.2, (with $\mathcal H^s$ normalized as in \cite{M5}),
\begin{equation}\label{dens}
\theta^{\ast s}(A,x)\leq 1\ \text{for}\ \mathcal H^s\ \text{almost all}\ x\in A.
\end{equation}



By the notation $M\lesssim N$ we mean that $M\leq CN$ for some constant $C$. The dependence of $C$ should be clear from the context. 
By  $c$ we mean positive constants with obvious dependence on the related parameters. 

\section{Dimension of level sets}\label{dimension}

We shall now prove Theorem \ref{levela}, the proof for Theorem \ref{level} is almost the same. The following argument follows very closely that of Harris in \cite{H} for line sections in the first Heisenberg group.

\begin{proof}[Proof of Theorem \ref{levela}]

Note first that using \eqref{dens} our assumptions imply that $P_{\lambda\sharp}(\mathcal H^s\restrict A)\ll\mathcal L^m$ for $\omega$ almost all $\lambda\in\Lambda$.

For any $\lambda\in \Lambda$ the inequality $\dim P_{\lambda}^{-1}\{u\}\cap A \leq s-m$ for $\mathcal L^m$ almost all $u\in\R^m$ follows for example from \cite{M5}, Theorem 7.7. This implies $\dim P_{\lambda}^{-1}\{P_{\lambda}x\}\cap A \leq s-m$ for $\mathcal H^s$ almost all $x\in A$ whenever $P_{\lambda\sharp}(\mathcal H^s\restrict A)\ll\mathcal L^m$. Hence we only need to prove the opposite inequalities. 

Define $\mu = 10^{-s}\mathcal H^s \restrict A$. Due to \eqref{dens} we may assume that $\mu(B(x,r))\leq r^s$ for $x\in \Rn, r>0$, by restricting $\mu$ to a suitable subset of $A$ with large measure. We may also assume that $A$ is compact, which makes it easier to verify the measurabilities. 

Let $0<t<s-m, 0<r<1$, and let $B_j=B(a_j,r)\subset\Rn,j=1,\dots ,j_1,$ be such that $B(0,1)\subset\cup_jB(a_j,r/2)$ and the balls $B_{j}, i=1,\dots,j_1,$ have bounded overlap, that is, there is an integer $N$, depending only on $n$, such that any point of $\Rn$ belongs to at most $N$ balls $B_{j}, i=1,\dots,j_1$. Let $0<\delta<r/2$ and let $\mu_j$ be the restriction of $\mu$ to $B_j$. Then we have

\begin{align*}
&\iint (r^{-t}\delta^{-m} \mu(\{y\in B(x,r/2):|P_{\lambda}(y-x)|\leq\delta\}))^{p-1}\,d\mu x\,d\omega\lambda\\
&\lesssim\sum_j\iint (r^{-t}\delta^{-m} \mu_j(\{y:|P_{\lambda}(y-x)|\leq\delta\}))^{p-1}\,d\mu_j x\,d\omega\lambda
\end{align*}

Here, when $P_{\lambda\sharp}\mu_j\ll\mathcal L^m$,

\begin{align*}
&\int (r^{-t}\delta^{-m} \mu_j(\{y:|P_{\lambda}(y-x)|\leq\delta\}))^{p-1}\,d\mu_j x\\\
&=\int (r^{-t}\delta^{-m} P_{\lambda\sharp}\mu_j(B(z,\delta)))^{p-1}\,dP_{\lambda\sharp}\mu_j z\\
&\lesssim r^{-t(p-1)}\int(M(P_{\lambda\sharp}\mu_j)(z))^{p}\,dz \lesssim r^{-t(p-1)}\|P_{\lambda\sharp}\mu_j\|_p^p,
\end{align*}
where $M(f)$ is the Hardy-Littlewood maximal function of $f$ and the last inequality follows from its $L^p$ boundedness.

It follows that
\begin{align*}
&\iint (r^{-t}\delta^{-m} \mu(\{y\in B(x,r/2):|P_{\lambda}(y-x)|\leq\delta\}))^{p-1}\,d\mu x\,d\omega\lambda\\
&\lesssim r^{-t(p-1)}\sum_j\int\|P_{\lambda\sharp}\mu_j\|_p^p\,d\omega\lambda.
\end{align*}

For $a,x\in\Rn, r>0,$ define $T_{a,r}(x)=(x-a)/r$ and let $\nu_j = r^{-s}T_{a_j,r\sharp}(\mu_j)\in\mathcal M(B(0,1))$.  Then one easily checks that  $\nu_j(B(x,\rho))\leq \rho^s$ for $x\in\Rn$ and $\rho>0$. Moreover, 
\begin{equation}\label{e1}
\|P_{\lambda\sharp}\mu_j\|_p^p = r^{m+p(s-m)}\|P_{\lambda\sharp}\nu_j\|_p^p.\end{equation}
To check this we may assume that $P_{\lambda}(x,y)=x$ for $x\in\R^m, y\in\R^{n-m}$. By approximation we may assume that $\mu_j$ is a continuous function. Then $P_{\lambda\sharp}\mu_j(x)=\int \mu_j(x,y)\,dy, T_{a_j,r\sharp}\mu_j(x,y)=r^n\mu_j((rx,ry)+a_j)$ and 
$$P_{\lambda\sharp}\nu_j(x)=r^{-s}P_{\lambda\sharp}T_{a_j,r\sharp}\mu_j(x,y)=r^{n-s}\int \mu_j((rx,ry)+a_j)\,dy,$$
from which \eqref{e1} follows by change of variable.
By \eqref{L2a},
$$\int\|P_{\lambda\sharp}\nu_j\|_p^p\,d\omega\lambda \lesssim \nu_j(B(0,1)) = r^{-s}\mu(B_j),$$ 
Hence by the bounded overlap,
\begin{align*}
&\iint (r^{-t}\delta^{-m} \mu(\{y\in B(x,r/2):|P_{\lambda}(y-x)|\leq\delta\}))^{p-1}\,d\mu x\,d\omega\lambda\\
&\lesssim r^{(p-1)(s-m-t)}\sum_j\mu(B_j)\lesssim r^{(p-1)(s-m-t)}\mu(B(0,1)).\end{align*}

Summing over $r=2^{-j}, j\geq k,$ and using the fact that $(p-1)(s-m-t)>0$, yields

\begin{align*}
&\sum_{j\geq k}\liminf_{\delta \to 0}\iint (2^{-jt}\delta^{-m} \mu(\{y\in B(x,2^{-j}):|P_{\lambda}(y-x)|\leq\delta\}))^{p-1}\,d\mu x\,d\omega\lambda\\
&\lesssim 2^{-k(p-1)(s-m-t))}\mu(B(0,1)).
\end{align*}

By the monotone convergence theorem and Fatou's lemma this gives
\begin{align*}
&\iint\limsup_{r\to 0}\liminf_{\delta \to 0} (r^{-t}\delta^{-m} \mu(\{y\in B(x,r):|P_{\lambda}(y-x)|\leq\delta\}))^{p-1}\,d\mu x\,d\omega\lambda\\
&=0.
\end{align*}
Hence for $\omega$ almost all $\lambda\in \Lambda$ and $\mu$ almost all $x\in A$, 
\begin{equation}\label{e8}
\lim_{r\to 0}\liminf_{\delta\to 0}r^{-t}\delta^{-m} \mu(\{y\in B(x,r):|P_{\lambda}(y-x)|\leq\delta\}) = 0.
\end{equation}

This is the same as (3.6) in \cite{M9} and after that the proof is essentially the same as that of Theorem \cite[Theorem 3.1]{M9}. 
\end{proof}

For an application to intersections we shall need the following product set version of Theorem \ref{levela}. There $P_{\lambda}:\Rn\times\R^l\to\R^m, \lambda\in\Lambda, m<n+l,$ are orthogonal projections with the same assumptions as before.

\begin{thm}\label{levelprod}
Let $s,t>0$ with $s+t>m$ and $p>1$. Suppose that $P_{\lambda\sharp}(\mu\times\nu)\ll\mathcal L^m$ for $\omega$ almost all $\lambda\in \Lambda$ and there exists a positive number $C$ such that

\begin{equation}\label{L3}
\iint P_{\lambda\sharp}(\mu\times\nu)(u)^p\,d\mathcal L^mu\,d\omega\lambda < C\mu(B^n(0,1))\nu(B^l(0,1))
\end{equation}
whenever $\mu\in\mathcal M(B^n(0,1)), \nu\in\mathcal M(B^l(0,1))$ are such that $\mu(B(x,r))\leq r^s$ for $x\in \Rn, r>0$, and $\nu(B(y,r))\leq r^t$ for $y\in \R^l, r>0$. 

If $A\subset\R^n$ is $\mathcal H^s$ measurable with $0<\mathcal H^s(A)<\infty$ and $B\subset\R^l$ is $\mathcal H^t$ measurable with $0<\mathcal H^t(B)<\infty$, then 
for $\mathcal H^s\times \mathcal H^t\times\omega$ almost all $(x,y,\lambda)\in A\times B\times\Lambda$,
\begin{equation}\label{pr1}
\dim P_{\lambda}^{-1}\{P_{\lambda}(x,y)\}\cap (A\times B) \geq s+t-m,
\end{equation}
and for $\omega$ almost all $\lambda\in \Lambda$,
\begin{equation}\label{pr2}
\mathcal L^m(\{u\in\R^m: \dim P_{\lambda}^{-1}\{u\}\cap (A\times B) \geq s+t-m\}) > 0.
\end{equation}
\end{thm}

\begin{proof}
The proof is essentially the same as that of Theorem \ref{levela}. I sketch the main steps. We apply the same argument to $\mu\times\nu=10^{-s}(\mathcal H^s\restrict A)\times(\mathcal H^t\restrict B)$ in place of $\mu=10^{-s}\mathcal H^s\restrict A$. We have again  $\mu(B(x,r))\leq r^s$ for $x\in \Rn, r>0$ and $\nu(B(y,r))\leq r^t$ for $y\in \R^l, r>0$. 

Let $0<\sigma<s+t-m, 0<r<1$, and let $B_{j,k}=B(a_j,r)\times B(b_k,r)\subset \R^n\times\R^l,j=1,\dots ,j_1,k=1,\dots ,k_1,$ be such that $B^n(0,1)\subset\cup_jB(a_j,r/2), B^l(0,1)\subset\cup_kB(b_k,r/2)$ and the sets $B_{j,k}, j=1,\dots,j_1, k=1,\dots ,k_1,$ have bounded overlap. Let $0<\delta<r/2$ and let $\mu_j$ be the restriction of $\mu$ to $B(a_j,r)$ and $\nu_k$ the restriction of $\nu$ to $B(b_k,r)$. Then

\begin{align*}
&\iint (r^{-\sigma}\delta^{-m} \mu\times\nu(\{(u,v)\in B(x,r/2)\times B(y,r/2):\\
&|P_{\lambda}((u,v)-(x,y))|\leq\delta\}))^{p-1}\,d\mu\times\nu (u,v)\,d\omega\lambda\\
&\lesssim\sum_{j,k}\iint (r^{-\sigma}\delta^{-m} \mu_j\times\nu_k(\{(u,v):|P_{\lambda}((u,v)-(x,y))|\leq\delta\}))^{p-1}\,d\mu_j\times\nu_k (x,y)\,d\omega\lambda.
\end{align*}
From this we conclude as before that

\begin{align*}
&\iint (r^{-\sigma}\delta^{-m} \mu\times\nu(\{(u,v)\in B(x,r/2)\times B(y,r/2):\\
&|P_{\lambda}((u,v)-(x,y))|\leq\delta\}))^{p-1}\,d\mu\times\nu (x,y)\,d\omega\lambda\\
&\lesssim r^{-\sigma(p-1)}\sum_{j,k}\int\|P_{\lambda\sharp}(\mu_j\times\nu_k)\|_p^p\,d\omega\lambda.
\end{align*}

Letting $\tilde{\mu}_j = r^{-s}T_{a_j,r\sharp}(\mu_j)\in\mathcal M(B^n(0,1)), \tilde{\nu}_k = r^{-t}T_{b_k,r\sharp}(\nu_k)\in\mathcal M(B^l(0,1))$ we have
\begin{equation}\label{e2}
\|P_{\lambda\sharp}(\mu_j\times\nu_k)\|_p^p = r^{m+p(s+t-m)}\|P_{\lambda\sharp}(\tilde{\mu}_j\times\tilde{\nu}_k)\|_p^p\end{equation}
which by \eqref{L3} leads to
\begin{align*}
&\iint (r^{-\sigma}\delta^{-m} \mu\times\nu(\{(u,v)\in B(x,r/2)\times B(y,r/2):\\
&|P_{\lambda}((u,v)-(x,y))|\leq\delta\}))^{p-1}\,d\mu\times\nu (u,v)\,d\omega\lambda\\
&\lesssim r^{-\sigma(p-1)}\sum_{j,k}\int\|P_{\lambda\sharp}(\mu_j\times\nu_k)\|_p^p\,d\omega\lambda\\
&\lesssim r^{(p-1)(s+t-m-\sigma)}\sum_j\mu\times\nu(B_{j,k})\lesssim r^{(p-1)(s+t-m-\sigma)}\mu(B^n(0,1))\nu(B^l(0,1))\end{align*}
and finally for $\omega$ almost all $\lambda\in \Lambda$ and $\mu\times\nu$ almost all $(x,y)\in A\times B$, 
\begin{align*}
&\lim_{r\to 0}\liminf_{\delta \to 0} r^{-\sigma}\delta^{-m} \mu\times\nu(\{(u,v)\in B(x,r)\times B(y,r):
|P_{\lambda}((u,v)-(x,y))|\leq\delta\})=0.
\end{align*}

Again, the rest of the proof is essentially the same as that of Theorem \cite[Theorem 3.1]{M9}; one applies \cite[Lemma 3.2]{M9} to $\mu\times\nu$ in place of $\mu$. 
\end{proof}

For the opposite direction we now have the inequality $\dim P_{\lambda}^{-1}\{u\}\cap (A\times B) \leq \dim A\times B-m$ for almost all $u\in\R^m$ by   \cite[Theorem 7.7]{M5}, but we often have $\dim A\times B > s+t$.

\section{Intersections}\label{intersections}

We now apply Theorem \ref{levelprod} to the projections $P_g, P_g(x,y)=x-g(y), x,y\in\Rn, g\in O(n),$ to prove Theorem \ref{inter} on the Hausdorff dimension of intersections. All we need to do is to check the estimate \eqref{L3}, after that the proof runs as that of \cite[Theorem 4.1]{M9}.  The qualitative version of \eqref{L3} was given in the proof of \cite[Theorem 4.2]{M8}. We just have to check that that argument yields the upper bound we need. For convenience, I give essentially the whole short proof.

\begin{proof}[Proof of Theorem \ref{inter}]


Let $\mu, \nu\in\mathcal M(B^n(0,1))$ with $\mu(B(x,r))\leq r^{s}, \nu(B(x,r))\leq r^{t}$ for $x\in\R^n, r>0$.  Set for $r>1$,

$$\sigma(\nu)(r)=\int_{S^{n-1}}|\widehat{\nu}(rv)|^2\,d\sigma^{n-1}v.$$ 
Let $0<t'<t$ with $s+(n-1)t'/n>n$. Then by the results of Du and Zhang, \cite{DZ}, Theorem 2.8,
\begin{equation}\label{dz}
\sigma(\nu)(r) \lesssim r^{-(n-1)t'/n}\nu(B(0,1)).
\end{equation} 
The factor $\nu(B(0,1))$ is not stated in  \cite{DZ}, but we shall check it below. For $n=2$ this estimate was proved by Wolff in \cite{W}, see \cite[Theorem 16.1]{M6} where this bound is explicitly stated.
To apply Theorem \ref{levelprod} we need that the implicit constant here is independent of $\nu$ as long as $\nu\in\mathcal M(B^n(0,1))$ and $\nu(B(x,r))\leq r^{t}$ for $x\in\R^n, r>0$. 

We now check \eqref{dz}, cf. the proof \cite[Proposition 16.3]{M6}.  For $R>1$, define $\nu_R$ by $\int f\,d\nu_R = \int f(Rx)\,d\nu x$. Then $\widehat{\nu_R}(x) = \widehat{\nu}(Rx).$ By \cite[Theorem 2.3]{DZ} (for the sphere instead of parabola) for every $\e>0$,
$$\int |\widehat{f\sigma^{n-1}}|^2\,d\nu_R \lesssim R^{-(n-1)t/n+\e}\|f\|_{L^2(S^{n-1)})}^2,$$
where the implicit constant depends on $\e$ but not on $\nu$. Hence by duality and the Schwartz inequality,
\begin{align*}
&\int_{S^{n-1}} |\widehat{\nu_R}|^2\,d\sigma^{n-1}=\sup_{\|f\|_{L^2(S^{n-1)})}=1}\left |\int_{S^{n-1}} \widehat{\nu_R}f\,d\sigma^{n-1}\right |^2\\
&=\sup_{\|f\|_{L^2(S^{n-1)})}=1}\left |\int_{S^{n-1}} \widehat{f\sigma^{n-1}}\,d\nu_R\right |^2\lesssim R^{-(n-1)t/n+\e}\nu_R(B(0,R)).
\end{align*}
By the definition of $\nu_R$ this is
$$\int_{S^{n-1}} |\widehat{\nu}(Rx)|^2\,d\sigma^{n-1}x \lesssim R^{-(n-1)t/n+\e}\nu(B(0,1)),$$
which is the required estimate.

As   $\widehat{P_{g\sharp}(\mu\times\nu)}(\xi)=\widehat{\mu}(\xi)\widehat{\nu}(-g^{-1}(\xi)), |\widehat{\mu}(\xi)|^2\lesssim \mu(B(0,1))^2\leq \mu(B(0,1))$ and similarly $\sigma(\nu)(r)\lesssim \nu(B(0,1))$, we have 
\begin{equation}\label{e7}
\begin{split}
&\iint|\widehat{P_{g\sharp}(\mu\times\nu)}(\xi)|^2\,d\xi\,d\theta_n g
=c\int\sigma(\nu)(|\xi|)|\widehat{\mu}(\xi)|^2\,d\xi\\
&\lesssim \mathcal L^n(B(0,1))\mu(B(0,1))\nu(B(0,1)) + \int_{|\xi|>1}|\widehat{\mu}(\xi)|^2|\xi|^{-(n-1)t'/n}\,d\xi\nu(B(0,1))\\
&=\mathcal L^n(B(0,1))\mu(B(0,1))\nu(B(0,1)) +c'I_{n-(n-1)t'/n}(\mu)\nu(B(0,1))\\
&\leq  C(n,s,t')\mu(B(0,1))\nu(B(0,1)).
\end{split}
\end{equation}
The easy estimate $I_{n-(n-1)t'/n}(\mu)\lesssim\mu(B(0,1))$ holds since $n-(n-1)t'/n<s$, see, for example, the last display on page 19 of \cite{M6}.  

We can now apply \eqref{pr1}\
 of Theorem \ref{levelprod}. It gives $\dim P_g^{-1}\{P_g(x,y)\}\cap (A\times B) \geq s+t-n$ for $\mathcal H^s\times\mathcal H^t\times\theta_n$ almost all $(x,y,g)\in A\times B\times O(n)$. Notice that $(u,v)\in P_g^{-1}\{P_g(x,y)\}\cap (A\times B)$ if and only if $u\in A, v\in B$ and $u=g(v-y)+x$, that is,
\begin{equation}\label{prodproj}
A\cap (g(B-y)+x)=\Pi(P_g^{-1}\{P_g(x,y)\}\cap (A\times B)),\end{equation}
where the projection $\Pi(x,y)=x$ is a constant times isometry on any $n$-plane $\{(u,v):u=g(v)+w\}$.
Hence \eqref{e5} follows. In the same way \eqref{e6} follows from \eqref{pr2} of Theorem \ref{levelprod}.
\end{proof}

Dimension estimates for sets of exceptional orthogonal transformations were obtained in \cite{M1}. The following theorem improves those estimates except when one of the dimensions $s$ and $t$ is at most $(n-1)/2$.

\begin{thm}\label{exinter}
Let $s$ and $t$ be positive numbers with $s+t > n$. If $A\subset\R^n$ is $\mathcal H^s$ measurable with $0<\mathcal H^s(A)<\infty$ and $B\subset\R^n$ is $\mathcal H^t$ measurable with $0<\mathcal H^t(B)<\infty$, then 
there is a Borel set $E\subset O(n)$ such that 
\begin{equation}\label{ex3} 
\dim E\leq n(n-1)/2-(s+(n-1)t/n-n), \end{equation}
and for  $g\in O(n)\setminus E$,
\begin{equation}\label{eq4} 
\mathcal L^n(\{z\in\Rn: \dim A\cap (g(B)+z)\geq s+t-n\})>0.
\end{equation}
\end{thm}

Note that $n(n-1)/2$ is the dimension of $O(n)$ so that \eqref{ex3} is relevant only when $s+(n-1)t/n>n$. In \cite{M1} the weaker estimate $\dim E\leq n(n-1)/2-(s+t-n-1)$ was derived. In the case $s$ or $t$ is at most $(n-1)/2$, then, essentially,  
\begin{equation*} 
\dim E\leq n(n-1)/2-(s+t-n). \end{equation*}
See \cite[Theorem 1.3]{M1} for the slightly weaker precise statement. I believe that the upper bounds $n(n-1)/2-(s+t-n)$, when $n\geq 3$, and, when $n=2$, $3-s-t$, if $s+t\leq 3$ and $0$, if $s+t\geq 3$,  should always be valid.

\begin{proof}[Proof of Theorem \ref{exinter}]
We may assume that $A$ and $B$ are compact. Then  the set $E$ of $g\in O(n)$ such that \eqref{eq4} fails is a Borel set. If $\dim E > n(n-1)/2-(s+(n-1)t/n-n)$, we can choose a compact subset $F$ of $E$ and $\a$ such that $\dim F>\a > n(n-1)/2-(s+(n-1)t/n-n)$. Then there is $\theta\in\mathcal M(F)$ such that $\theta(B(g,r))\leq r^{\a}$ for $g\in O(n), r>0$. Let $\mu, \nu\in\mathcal M(B^n(0,1))$ be such that $\mu(B(x,r))\leq r^s$ and $\nu(B(x,r))\leq r^t$ for $x\in \Rn, r>0$.  

We shall apply Theorem \ref{levelprod} with $\Lambda=F$ and $P_g(x,y)=x-g(y)$. By the proof of \cite[Theorem 4.3]{M8},
$$\iint|\widehat{P_{g\sharp}(\mu\times\nu)}(\xi)|^2\,d\xi\,d\theta g< C\mu(B^n(0,1))\nu(B^l(0,1)),$$
with $C$ independent of $\mu$ and $\nu$. The right hand side is not explicitly stated in \cite{M8}, but it can be checked in the same way as in the proof of Theorem \ref{inter}. Then, as in the proof of Theorem \ref{inter}, it follows from Theorem \ref{levelprod} and the formula \eqref{prodproj} that 
\eqref{eq4} holds for $\theta$ almost all $g\in F$. This contradiction completes the proof.
\end{proof}

\section{Some comments}
For the applications to intersections we only used Theorem \ref{levela} with $p=2$, which, essentially, was already proved in \cite{M9}. For  a similar method  Harris \cite{H} also needed the case $p<2$; he proved the $L^{3/2}$ estimate for vertical projections in the first Heisenberg group and derived from it the almost sure dimension of line sections.

Orponen proved in \cite{O3} an $L^p$ estimate for radial projections with some $p>1$. If the analog of Theorem \ref{level} or \ref{levela} would be true in this setting, it would  
improve some of the results of \cite{MO}. Although most of the arguments for Theorems \ref{level} and \ref{levela} work much more generally, the scaling argument, in particular \eqref{e1}, seems to require right kind of scaling properties of the maps. 

\vspace{1cm}
\begin{footnotesize}
{\sc Department of Mathematics and Statistics,
P.O. Box 68,  FI-00014 University of Helsinki, Finland,}\\
\emph{E-mail address:} 
\verb"pertti.mattila@helsinki.fi" 

\end{footnotesize}

\end{document}